\RequirePackage{etex}
\RequirePackage{ifpdf}
\ifpdf
\documentclass[12pt]{amsart}
\else
\documentclass[a4paper,dvipdfmx,12pt]{amsart}
\fi

\usepackage[top=30truemm,bottom=30truemm,left=30truemm,right=30truemm]{geometry}

\usepackage{iftex}
\ifLuaTeX
\RequirePackage{luatex85}
\fi

\usepackage[foot]{amsaddr}
\usepackage[nocompress,noadjust]{cite}

\usepackage{tikz}
\usepackage{tikz-cd}
\usetikzlibrary{cd}
\usetikzlibrary{positioning, arrows.meta, matrix}
\usetikzlibrary{decorations.pathreplacing}
\tikzcdset{
  arrow style=tikz,
  arrows=-stealth
}

\usepackage{lipsum}
\usepackage{adjustbox}

\usepackage{fancybox}

\usepackage{listings}
\usepackage{docmute}
\usepackage{amssymb, enumerate, calligra}
\usepackage{amsthm}
\usepackage{amsmath}
\usepackage{amsxtra}
\usepackage{latexsym}
\usepackage{mathrsfs}
\usepackage[all,cmtip]{xy}
\usepackage[all]{xy}
\usepackage{enumitem}
\usepackage{xcolor}
\usepackage{graphicx}
\usepackage{comment}
\usepackage{mathabx,epsfig}
\usepackage{stmaryrd}

\usepackage[hypertexnames=false,hyperindex,breaklinks]{hyperref}
\usepackage[nameinlink]{cleveref}
\usepackage{autonum}

\usepackage{mathtools}

\makeatletter
\renewcommand\th@plain{\slshape}
\makeatother
\newtheoremstyle{plain}
 {2mm}
 {2mm}
 {\slshape}
 {}
 {\bfseries}
 {.}
 {.5em}
 {}
\theoremstyle{plain}
\newtheorem{theorem}{Theorem}[section]

\newtheorem{lemma}[theorem]{Lemma}
\newtheorem{proposition}[theorem]{Proposition}
\newtheorem{claim}[theorem]{Claim}
\newtheorem*{claim*}{Claim}

\newtheorem{question}[theorem]{Question}

\newtheoremstyle{definition}
 {2mm}
 {2mm}
 {\normalfont}
 {}
 {\bfseries}
 {.}
 {.5em}
 {}
\theoremstyle{definition}

\newtheorem{definition}[theorem]{Definition}

\newtheorem{remark}[theorem]{Remark}
\newtheorem{example}[theorem]{Example}
\newtheorem*{acknowledgements}{Acknowledgements}

\crefname{section}{Section}{Sections}

\crefname{theorem}{Theorem}{Theorems}
\crefname{corollary}{Corollary}{Corollaries}
\crefname{lemma}{Lemma}{Lemmas}
\crefname{lemma}{Lemma}{Lemmas}
\crefname{proposition}{Proposition}{Propositions}
\crefname{claim}{Claim}{Claims}
\crefname{definition}{Definition}{Definitions}
\crefname{notation}{Notation}{Notations}
\crefname{problem}{Problem}{Problems}
\crefname{note}{Note}{Notes}
\crefname{remark}{Remark}{Remarks}
\crefname{example}{Example}{Examples}
\crefname{conjecture}{Conjecture}{Conjectures}
\crefname{question}{Question}{Questions}
\crefname{mainthm}{Theorem}{Theorems}
\crefname{mainprop}{Proposition}{Propositions}

\crefname{enumi}{}{}
\crefname{enumii}{}{}
\crefname{enumiii}{}{}

\crefformat{equation}{{\upshape(#2#1#3)}}

\numberwithin{equation}{section}

\def\Q{{\mathbb Q}}
\def\R{{\mathbb R}}
\def\Z{{\mathbb Z}}
\def\P{{\mathbb P}}
\def\A{{\mathbb A}}

\def\QQ{\overline{\mathbb Q}}
\def\p{{ \mathfrak{p}}}     
\def\q{{ \mathfrak{q}}}
\def\dd{{ \mathfrak{d}}}     
\def\O{{ \mathcal{O}}}
 
\def\a{{ \mathfrak{a}}}  
  
\def\I{{ \mathcal{I}}}
\def\J{{ \mathcal{J}}}

\def\L{{ \mathcal{L}}}
\def\M{{ \mathcal{M}}}

\DeclareMathOperator{\rank}{rank}

\DeclareMathOperator{\Spec}{Spec}
\DeclareMathOperator{\Proj}{Proj}

\DeclareMathOperator{\codim}{codim} 
\DeclareMathOperator{\vol}{vol}

\DeclareMathOperator{\Bl}{Bl}

\renewcommand{\a}{\alpha}

\renewcommand{\b}{\beta}

\newcommand{\g}{\gamma}
\renewcommand{\d}{\delta}
\newcommand{\e}{\varepsilon}

\renewcommand{\l}{\lambda}
\renewcommand{\k}{\kappa}

\newcommand{\G}{\Gamma}

\allowdisplaybreaks

\begin{document}

	\title[Growth of generalized gcd]
	{Growth of generalized greatest common divisors along orbits of self-rational maps on projective varieties} 
	
	\author[Yohsuke Matsuzawa]{Yohsuke Matsuzawa}
	\address{Department of Mathematics, Graduate School of Science, Osaka Metropolitan University, 3-3-138, Sugimoto, Sumiyoshi, Osaka, 558-8585, Japan}
	\email{matsuzaway@omu.ac.jp}

	\begin{abstract} 
		Consider a dominant rational self-map $f$ on a smooth projective variety $X$ defined over $\QQ$.
		We prove that 
		\begin{align}
		\lim_{n \to \infty} \frac{h_{Y}(f^{n}(x))}{h_{H}(f^{n}(x)) } = 0,
		\end{align}
		where $h_{Y}$ is a height 
		associated with a closed subscheme $Y \subset X$ of codimension $c$, $h_{H}$ is any ample height on $X$, and 
		$x \in X(\QQ)$ is a point with well-defined orbit, under the following assumptions:
		(1) either $f$ is a morphism, or $Y$ is pure dimensional, regularly embedded in $X$, and contained in 
		the locus over which all iterates of $f$ are finite;
		(2) the orbit of $x$ is generic;
		(3) $d_{c}(f)^{1/c} < \alpha_{f}(x)$,
		where $d_{c}(f)$ is the $c$-th dynamical degree of $f$ and $ \alpha_{f}(x)$ is the arithmetic degree of $x$.
	\end{abstract}

	

	\maketitle
	\tableofcontents

\section{Introduction}\label{sec:intro}

Let $X$ be a smooth projective variety, 
and $f \colon X \dashrightarrow X$ a dominant rational map, both defined over $\QQ$.
For a proper closed subscheme $Y \subset X$, we can define associated global height function $h_{Y}$ (\cite{Sil87subht}),
which is nothing but a Weil height function when $Y$ is a Cartier divisor.
When $Y$ has codimension at least two, the case we are interested in, 
$h_{Y}$ is sometimes called generalized greatest common divisor.
This name may be justified by the following example.
Let $Y = \{(0:0:1)\} \subset \P^{2}_{\QQ}$ and $x = (a:b:c) \in \P^{2}(\QQ)$
with $a,b,c \in \Z$ coprime. Then
\begin{align}
h_{Y}(x) = \log \frac{\max\{|a|, |b|, |c|\}}{ \max\{|a|, |b|\}} + \log \gcd(a,b).
\end{align}

For a point $x \in X(\QQ)$ with well-defined forward $f$-orbit, 
$h_{Y}(f^{n}(x))$ is expected to be very small.
More precisely, given any Weil height function $h_{H}$ associated with an ample divisor $H$ on $X$,
we expect
\begin{align}
\lim_{n \to \infty} \frac{h_{Y}(f^{n}(x))}{h_{H}(f^{n}(x))} = 0 \label{eq:hYhHlimzero}
\end{align}
unless there are any special (geometric) reasons (cf.\ \cite[section 19]{silverman2024dynamical}).
Roughly speaking, this claims that the number of digits of greatest common divisors 
among coordinates of orbits are smaller than the number of digits of coordinates.
There are works \cite{matsuzawa2020vojta,Hu20} that directly study this problem,
but they all assume Vojta's conjecture.
We remark that Silverman pointed out \cite[Example 35]{silverman2024dynamical} that a famous theorem by Bugeaud-Corvaja-Zannier \cite[Theorem 1]{BCZ03}
can be interpreted dynamically.
For example, let $a,b \geq 2$ be multiplicatively independent integers, and 
$f \colon \P^{2}_{\QQ} \longrightarrow \P^{2}_{\QQ}, (x:y:z) \mapsto (ax:by:z)$.
Let $Y = \{(1:1:1)\}$.
Then 
\begin{align}
&\frac{h_{Y}(f^{n}(1:1:1))}{h_{H}(f^{n}(1:1:1))}\\
&=
\frac{1}{\log \max\{|a^{n}|, |b^{n}|, 1\}} 
\bigg(  
\log \frac{\max\{|a^{n}|, |b^{n}|, 1\}}{\max\{ |a^{n}-1|, |b^{n}-1| \}} 
+
\log \gcd( a^{n}-1, b^{n}-1 )
\bigg)
\end{align}
and \cite[Theorem 1]{BCZ03} is equivalent to say that this goes to $0$ as $n \to \infty$.
The key input to the proof of their theorem was Schmidt subspace theorem.

In this paper, we prove a sufficient condition for \cref{eq:hYhHlimzero}.
Our theorem is unconditional, i.e.\ does not assume Vojta's conjecture.
Our strategy is to take advantage of dynamical complexity, and so only applicable to
maps $f$ with non-trivial first dynamical degree. 
Thus we cannot deduce Bugeaud-Corvaja-Zannier's theorem from our one;
the map $f(x:y:z)=(ax:by:z)$ has trivial first dynamical degree $1$.

To state our theorem, let us introduce some notions.

\begin{definition}[Orbits under rational maps]
Let $f \colon X \dashrightarrow X$ be a dominant rational map on a projective variety 
(variety is irreducible and reduced in this paper) defined over an algebraically closed field $k$.
The indeterminacy locus of $f$ is denoted by $I_{f}$.
For a point $x \in X(k)$, we say the forward $f$-orbit is well-defined if
$f^{n}(x) \not\in I_{f}$ for all $n \geq 0$.
The set of $x \in X(k)$ with well-defined forward $f$-orbit is denoted by $X_{f}(k)$.
For $x \in X_{f}(k)$, $O_{f}(x) = \{x, f(x),f^{2}(x),\dots\}$ is called the orbit.
We say the orbit is generic if $O_{f}(x)$ is infinite and $O_{f}(x) \cap Z$ is finite for 
any proper Zariski closed subset $Z \subset X$.
\end{definition}

We recall the definition of dynamical degrees.
There are number of papers that study properties of dynamical degrees.
We refer \cite{Tr20,Da20} for basic properties of dynamical degrees.
(Note also that the recursive inequalities proven in \cite{xie2024algebraic}
imply important properties of dynamical degrees, such as convergence, log concavity, and product formulas.)

\begin{definition}[Dynamical degrees]
Let $f \colon X \dashrightarrow X$ be a dominant rational map on a projective variety 
defined over an algebraically closed field $k$.
Let $H$ be an ample divisor on $X$.
For $i \in \{0,\dots, \dim X\}$, the $i$-th dynamical degree of $f$ is 
\begin{align}
d_{i}(f) = \lim_{n \to \infty} ( \pi_{2,n}^{*}H^{i} \cdot  \pi_{1,n}^{*}H^{\dim X - i} )^{1/n}
\end{align}
where $\pi_{1,n}, \pi_{2,n}$ are the projections from the graph of $f^{n}$:
\begin{equation}
\begin{tikzcd}
\G_{f^{n}} \arrow[d,"\pi_{1,n}",swap] \arrow[rd, "\pi_{2,n}"] & \\
X \arrow[r, "f^{n}", dashed]& X.
\end{tikzcd}
\end{equation}
Note that $d_{\dim X}(f)$ is simply the degree of the function field extension induced by $f$.
\end{definition}

In \cite[Appendix B.1]{matsuzawa-note-ad-dls}, for a dominant rational map $f \colon X \dashrightarrow X$
on a projective variety, we introduced a subset of $X$ over which all iterates of $f$ are ``finite".
We include here the definition, but
refer \cite[Appendix B.1]{matsuzawa-note-ad-dls} for basic properties of this set.
\begin{definition}[Iteratively finite locus]
Let $f \colon X \dashrightarrow X$ be a dominant rational map on a projective variety 
defined over a field $k$.
\begin{enumerate}
\item
Let $\G_{f}$ be the graph of $f$ and $p_{i} \colon \G_{f} \longrightarrow X$ for $i=1,2$ be the projections:
\begin{equation}
\begin{tikzcd}
\G_{f} \arrow[d,"p_{1}", swap] \arrow[rd, "p_{2}"] & \\
X \arrow[r, "f", dashed]& X.
\end{tikzcd}
\end{equation}
We write $X_{f}^{\rm fin}$ the largest open subset of $X$ such that
\begin{align}
\text{$p_{2} \colon p_{2}^{-1}(X_{f}^{\rm fin}) \longrightarrow X_{f}^{\rm fin}$ is finite 
and $p_{2}^{-1}(X_{f}^{\rm fin})\cap p_{1}^{-1}(I_{f}) =  \emptyset$.}
\end{align}
Note that $f|_{X \setminus I_{f}}^{-1}(X_{f}^{\rm fin})$ is isomorphic to $p_{2}^{-1}(X_{f}^{\rm fin})$
via $p_{1}$.
\item
We inductively define $X_{f}^{{\rm fin}, n}$ for $n \geq 0$ as follows.
We set $X_{f}^{{\rm fin}, 0} = X$.
For $n \geq 1$, we set
\begin{align}
X_{f}^{{\rm fin}, n} = X_{f}^{{\rm fin}, n-1} \setminus f^{n-1}(f^{-(n-1)}(X_{f}^{{\rm fin}, n-1}) \setminus X_{f}^{{\rm fin}}).
\end{align}
Here $f^{-(n-1)}(X_{f}^{{\rm fin}, n-1})$ is $f|_{X \setminus I_{f^{n-1}}}^{-(n-1)}(X_{f}^{{\rm fin}, n-1})$ by definition.
By the construction, this is finite over $X_{f}^{{\rm fin}, n-1}$.
Note also that $X_{f}^{{\rm fin}, 1} = X_{f}^{{\rm fin}}$.
\item
We define
\begin{align}
X_{f}^{\rm back} = \bigcap_{n \geq 0} X_{f}^{{\rm fin}, n}.
\end{align}
\end{enumerate}
\end{definition}

When the dynamical system is defined over $\QQ$,
the growth rate of Weil height function along the orbit of a given point
is called arithmetic degree. 
\begin{definition}[Arithmetic degree]
Let $f \colon X \dashrightarrow X$ be a dominant rational map on a projective variety 
defined over $\QQ$.
Let $H$ be an ample divisor on $X$ and $h_{H}$ be an associated Weil height function on $X(\QQ)$.
For a point $x \in X_{f}(\QQ)$, 
\begin{align}
\alpha_{f}(x) = \lim_{n \to \infty} \max\{1, h_{H}(f^{n}(x)) \}^{1/n}
\end{align}
is called the arithmetic degree of $x$ provided the limit exists.
This is independent of the choice of $H$ and $h_{H}$ (even the limsup and liminf are independent).
If $f$ is a morphism \cite{KS16a} or $O_{f}(x)$ is generic \cite{matsuzawa-note-ad-dls}, then it is known that the limit exists. 
\end{definition}
For the basic properties of arithmetic degree, see \cite{matsuzawa2023recent}
and references therein.

Now we state our main theorem.

\begin{theorem}[Main theorem]\label{thm:gcdgrowth}
Let $X$ be a smooth projective variety of dimension $N$ over $\QQ$.
Let $f \colon X \dashrightarrow X$ be a dominant rational map.
Let $Y \subset X$ be a proper closed subscheme of dimension $l$.
If $f$ is not a morphism, assume further that $Y$ is pure dimensional,
the inclusion $Y \subset X$ is a regular embedding, and $Y \subset X_{f}^{\rm back}$.
Then for any $x \in X_{f}(\QQ)$, if
\begin{itemize}
\item
$O_{f}(x)$ is generic;
\item
$d_{N-l}(f)^{1/(N-l)} < \alpha_{f}(x)$,
\end{itemize}
then for any Weil height function $h_{H}$ associated with an ample divisor $H$ on X, we have
\begin{align}
\lim_{n \to \infty}\frac{h_{Y}(f^{n}(x))}{h_{H}(f^{n}(x))} = 0.
\end{align}
\end{theorem}
Recall that a closed immersion $Y \subset X$ is called regular embedding
if for any $x \in Y$, the ideal $I \subset \O_{X,x}$ of $Y$ at $x$ is generated by
a regular sequence.

\begin{remark}
We always have $ \alpha_{f}(x) \leq d_{1}(f)$ (for morphisms \cite{KS16a}, for rational maps \cite{Ma20a}).
Thus it is necessary to have $d_{N-l}(f)^{1/(N-l)} < d_{1}(f)$ so that 
\cref{thm:gcdgrowth} has content.
By the log concavity of dynamical degrees, $d_{N-l}(f)^{1/(N-l)} \leq d_{1}(f)$ always holds.
The equality holds for polarized endomorphisms 
(e.g.\ non-isomorphic surjective self-morphisms on projective spaces) for example, 
and thus we unfortunately can not apply our theorem to such morphisms.
\end{remark}

\begin{remark}
The Kawaguchi-Silverman conjecture predicts that $ \alpha_{f}(x) = d_{1}(f)$
for $x \in X_{f}(\QQ)$ with $O_{f}(x)$ being Zariski dense.
Also, the dynamical Mordell-Lang conjecture says if the orbit is Zariski dense, then it is generic.
Therefore, if these conjectures hold, the condition ``$O_{f}(x)$ is generic and $d_{N-l}(f)^{1/(N-l)} < \alpha_{f}(x)$"
is equivalent to ``$O_{f}(x)$ is Zariski dense and $d_{N-l}(f)^{1/(N-l)} < d_{1}(f)$".
\end{remark}

\begin{remark}
Without the condition $Y \subset X_{f}^{\rm back}$, the theorem has a counter-example.
See \cref{ex:backnonfinCE}.
\end{remark}

\begin{remark}
Compare \cref{thm:gcdgrowth} with \cite[Theorem 1.6]{matsuzawa2020vojta}.
When $f$ is a morphism, \cite[Theorem 1.6]{matsuzawa2020vojta} says Vojta's conjecture 
implies $h_{Y}(f^{n}(x))/h_{H}(f^{n}(x)) \to 0$ if  $\codim Y \geq 2$, $O_{f}$ is generic, and
$e(Y) < \alpha_{f}(x)$, where $e(Y)$ is a quantity that measures how $f^{n}$ ramifies 
along preimages of $Y$.
Although, there is no logical implications between these two theorems,
\cite[Theorem 1.6]{matsuzawa2020vojta} can be applicable in a meaningful way to endomorphisms on projective spaces.
\end{remark}

One of the key ingredients of the proof is to find an upper bound of
$h_{f^{-k}(Y)}$ for large $k$.
The idea is as follows. Let $\I$ be the ideal sheaf of $f^{-k}(Y) \subset X$.
If $H^{0}(X, \I^{r} \otimes \O_{X}(mH)) \neq 0$ for some $r, m \geq 1$,
then any non-zero global section would correspond to an effective divisor $D$ linearly equivalent to $mH$
and containing the closed subscheme defined by $\I^{r}$.
Then we get an inequality
\begin{align}
rh_{f^{-k}(Y)} \leq h_{D} + O(1) = mh_{H} + O(1)
\end{align}
on $(X \setminus D)(\QQ)$.
Thus our task is to figure out how small $m/r$ can be for such $r,m$. 
This idea of bounding generalized gcd is appeared in the proof of \cite[Theorem 3.1]{GP24}.
They give an upper bound of generalized gcd associated with reduced zero dimensional subschemes 
of smooth projective varieties (their theorem does not involve any dynamics).
When $\dim Y =0$, even though $f^{-k}(Y)$ is not necessarily reduced, their proof with
minor modification would give a sufficient upper bound for our purpose.
However, when $\dim Y > 0$, the situation is different and 
the problem becomes highly non-trivial. 
The main body of this paper is devoted to treat this problem.
At this moment, we do not know any clean upper bound of given $h_{Y}$ in general.  
Nevertheless, taking advantage of complexity of dynamical systems,  
we can get a simple upper bound of $h_{f^{-k}(Y)}$ as in \cref{prop:ggcdpullbackbound}.

We remark that there is a paper \cite{BARRIOS_2025}
that claims a generalization of \cite[Theorem 3.1]{GP24} to higher dimensional subvarieties.
However, the key lemma \cite[Lemma 2.1]{BARRIOS_2025} has counter-examples (see below), and hence
we are not sure if the proof of the claimed upper bound of generalized gcd \cite[Theorem 1.2]{BARRIOS_2025}
is correct.
That said, we would like to note that the idea in \cite{BARRIOS_2025} served as a useful inspiration for our approach to obtaining the present result.

\begin{example}
Let $\pi \colon S = \Bl_{P}\P^{2} \longrightarrow \P^{2}$
be the blow up of $\P^{2}$ at a closed point $P$, everything over an algebraically closed field.
Let $E \subset S$ be the exceptional divisor and set $Z = mE$ with $m \geq 2$.
Then by Riemann-Roch, we have
\begin{align}
\chi(\O_{Z}) &= - \frac{1}{2} (mE \cdot (mE + K_{S})) \\
&= - \frac{1}{2} (mE \cdot (mE + \pi^{*}K_{\P^{2}} +E)) = \frac{m(m+1)}{2}.
\end{align}
Thus by Riemann-Roch for singular curves (cf.\ \cite[Theorem 3.17]{Liu}), for any ample line bundle $ \mathcal{A}$, we have
\begin{align}
\dim H^{0}(Z, \mathcal{A}) = \deg \mathcal{A} + \frac{m(m+1)}{2} + \dim H^{1}(Z, \mathcal{A}) > \deg \mathcal{A} + 1.
\end{align}
\end{example}

\noindent
{\bf convention}
\begin{itemize}
\item A  \emph{variety} over $k$ is a separated scheme of finite type over $k$ which is irreducible and reduced;
\item
For a number field $K$, let $M_{K}$ denote the set of absolute values normalized as in \cite[p11 (1.6)]{BG06}.
Namely, Namely, if $K=\Q$, then $M_{\Q}=\{|\ |_{p} \mid \text{$p=\infty$ or a prime number} \}$ with
\begin{align}
 &|a|_{\infty} =
 \begin{cases}
 a \quad \text{if $a\geq0$}\\
 -a \quad \text{if $a<0$}
 \end{cases}
 \\
 & |a|_{p} = p^{-n} \quad \txt{if $p$ is a prime and  $a=p^{n}\frac{k}{l}$ where\\ $k,l$ are non zero integers coprime to $p$.}
\end{align}
For a number field $K$, $M_{K}$ consists of the following absolute values:
\begin{align}
|a|_{v} = |N_{K_{v}/\Q_{p}}(a)|_{p}^{1/[K:\Q]} 
\end{align}
where $v$ is a place of $K$ which restricts to $p = \infty$ or a prime number.
We use this set of absolute values to define height functions associated with subschemes.

\item
Let $X$ be a projective variety over a number field $K$.
Let $Y \subset X$ be a closed subscheme.
A local height function associated with $Y$ is denoted by $\l_{Y,v}$ or $\{ \l_{Y,v}\}_{v \in M_{K}}$.

\item
Let $X$ be a projective variety over $\QQ$.
Let $D$ be a Cartier divisor on $X$.
Let $Y \subset X$ be a closed subscheme.
A Weil height function associated with $D$ is denoted by $h_{D}$,
and similarly, a global height function associated with $Y$ is denoted by $h_{Y}$.
These are determined up to bounded function. When we write $h_{D}$ or $h_{Y}$
without mentioning anything, this means we choose a function from the up to bounded equivalence class. 
Unless $Y$ is an effective Cartier divisor, $h_{Y}$ is defined only on $(X \setminus Y)(\QQ)$.
As usual, we set $h_{Y}= \infty$ on $Y(\QQ)$.

\end{itemize}
For basic properties of height functions, \cite{BG06,HS00,La83} are standard references.
For height associated with subschemes, see the original article \cite{Sil87subht}, or \cite{matsuzawa2020height} for example.

\begin{acknowledgements}
This work has been done while the author's research stay at Harvard University.
We thank their hospitality, especially Laura DeMarco for hosting the author.
The author is supported by JSPS KAKENHI Grant Number JP22K13903 and 23KK0252.
\end{acknowledgements}

\section{Upper bound of generalized gcd and proof of the main theorem}

The goal of this section is to prove \cref{thm:gcdgrowth}.
As explained in the introduction, the idea to find an upper bound of $h_{Y}$ in terms of $h_{H}$
is as follows. Let $\I$ be the ideal sheaf of $Y \subset X$.
If $H^{0}(X, \I^{r} \otimes \O_{X}(m H) ) \neq 0$, we get $r h_{Y} \leq mh_{H} + O(1)$ on a Zariski open set of $X$.
By the following exact sequence 
\begin{align}
0 \longrightarrow \I^{r} \otimes \O_{X}(m H) \longrightarrow \O_{X}(mH) \longrightarrow \O_{X}(mH) \otimes \O_{X} / \I^{r} \longrightarrow 0,
\end{align}
we have
\begin{align}
&\dim H^{0}(X, \I^{r} \otimes \O_{X}(m H) ) \\
&\geq \dim H^{0}(X, \O_{X}(mH)) - \dim H^{0}(X, \O_{X}(mH) \otimes \O_{X} / \I^{r}).
\end{align}
The growing rate of $\dim H^{0}(X, \O_{X}(mH))$ when $m \to \infty$ is well understood: 
it is controlled by the self-intersection number $(H^{\dim X})$.
The first task is to find an upper bound of $\dim H^{0}(X, \O_{X}(mH) \otimes \O_{X} / \I^{r})$ for varying $r$ and $m$.
Then we apply it to $f^{-n}(Y)$ in place of $Y$ and relate the upper bound with dynamical degrees of $f$.

\vspace{1.5em}

Let $X$ be a projective variety over a field $k$
and $ \mathcal{L}$ be an invertible $\O_{X}$-module.
The complete linear system associated with $ \mathcal{L}$ is denoted by
$| \mathcal{L}|$. (It is $(H^{0}(X, \mathcal{L}) \setminus \{0\})/k^{\times}$ as a set and we equip it with the Zariski topology.)
Let  $ \dd \subset | \mathcal{L}|$ be a sub linear system and $l \in \Z_{\geq 0}$.
We say a property $P$ of tuple $(H_{1},\dots , H_{l}) \in \dd^{l}$ holds for 
a general sequence $H_{1}, \dots , H_{l} \in \dd$ when the following holds.
There are families of non-empty Zariski open subsets of $\dd$
\begin{align}
U_{1}, \{ U_{2, H_{1}} \}_{H_{1} \in U_{1}}, \dots, 
\{ U_{l, H_{1},\dots, H_{l-1}}\}_{(H_{1}, \dots, H_{l-1}) \in U_{1} \times U_{2, H_{1}} \times \cdots \times U_{l-1, H_{1},\dots, H_{l-2}}}
\end{align} 
such that if a sequence $(H_{1},\dots , H_{l}) \in \dd^{l}$ satisfies
\begin{align}
H_{1} \in U_{1}, H_{2} \in U_{2, H_{1}}, \dots, H_{l} \in U_{l, H_{1}, \dots, H_{l-1}},
\end{align}
then $P$ holds.
Note that if $k$ is infinite, then finite intersection of non-empty Zariski open subsets of $\dd$
is non-empty. Therefore, when $k$ is infinite, if there are properties $P_{1}, \dots, P_{s}$ of tuples $(H_{1},\dots , H_{l}) \in \dd^{l}$
that hold for a general sequence $H_{1}, \dots , H_{l} \in \dd$, all of them simultaneously holds for a general sequence $H_{1}, \dots , H_{l} \in \dd$.

For a projective scheme $X$ over a field $k$ and a coherent sheaf $ \mathcal{F}$,
we write $h^{0}(X, \mathcal{F})= \dim_{k}H^{0}(X, \mathcal{F})$, the dimension as a $k$-linear space.

\vspace{1em}

\begin{lemma}\label{lem:glsecthick}
Let $k$ be an infinite field.
Let $X$ be a projective variety over $k$ of dimension $N \geq 0$.
Let $Y \subset X$ be a proper closed subscheme of pure dimension $l$ with ideal sheaf $ \I \subset \O_{X}$.
Let $ \mathcal{L}$ be a very ample invertible $\O_{X}$-module on $X$
and $\dd \subset |\L|$ be a linear system that induces a closed immersion.
Then for a general sequence $H_{1}, \dots, H_{l} \in \dd$, we have
$H_{1\cdots i}:= H_{1} \cap \dots \cap H_{i}$ is irreducible and reduced of dimension $N-i$,
$Y \cap H_{1\cdots i}$ is pure dimension $l - i$, and
\begin{align}
h^{0}(X, \mathcal{L}^{\otimes m} \otimes \O_{X}/\I^{r}) \leq \sum_{i=0}^{l} \binom{m+i-1}{i} h^{0}(H_{1\cdots i}, \O_{H_{1\cdots i}}/\I_{i}^{r})
\end{align}
for all $m, r \in \Z_{\geq 1}$,
where $\I_{i} \subset \O_{H_{1\cdots i}}$ is the ideal sheaf of $Y \cap H_{1\cdots i} \subset H_{1\cdots i}$.
Note that $H_{1\cdots i} = X$ when $i=0$ by convention.
\end{lemma}

We prepare some lemmas.

\begin{lemma}\label{lem:quotreg}
Let $A$ be a Noetherian ring and $I \subset A$ be an ideal. 
Then there are finitely many prime ideals $\p_{1}, \dots, \p_{n} \subset A$
such that for all $a \in A \setminus (\p_{1} \cup \cdots \cup \p_{n}) $,
$a$ is $A/I^{r}$-regular for all $r \geq 1$.
\end{lemma}
\begin{proof}
When $I = A$, all $a \in A$ is $A/I^{r}$-regular and we are done.
Suppose $I \neq A$.
Let $R = \oplus_{r\geq 0} I^{r}/I^{r+1}$.
Then $R$ has a natural $A$-algebra structure and it is finitely generated over $A$.
Thus $R$ is also Noetherian ring.
Let $\p_{1}, \dots , \p_{n} \subset A$ be the inverse images of associated primes of $R$.
We claim that for all $a \in A \setminus (\p_{1} \cup \cdots \cup \p_{n}) $, $a$ is $A/I^{r}$-regular for all $r \geq 1$.
First note that $a$ is $A/I$-regular since $A/I \subset R$ and $a$ is $R$-regular by definition of $\p_{1}, \dots, \p_{n}$.
Let $r \geq 1$ and suppose $a$ is $A/I^{r}$ regular.
Consider the exact sequence
\begin{equation}
\begin{tikzcd}
0 \arrow[r] & I^{r}/I^{r+1} \arrow[r] & A/I^{r+1} \arrow[r] & A/I^{r} \arrow[r] & 0.
\end{tikzcd}
\end{equation}
Multiplication by $a$ is injective on $I^{r}/I^{r+1}$ since $a$ is $R$-regular.
It is injective on $A/I^{r}$ by assumption. Thus it is injective on $A/I^{r+1}$ by snake lemma.
\end{proof}

\begin{lemma}\label{lem:generalhpsecexact}
Let $X$ be a projective variety over an infinite field $k$.
Let $ \mathcal{L}$ be an invertible $\O_{X}$-module and
$\dd \subset | \mathcal{L}|$ be a base point free linear system.
Let $\I \subset \O_{X}$ be a coherent ideal sheaf.
Then for general $H \in \dd$, 
\begin{equation}
\begin{tikzcd}[column sep=small] 
0 \arrow[r] & \mathcal{L}^{\otimes m-1} \otimes \O_{X}/\I^{r} \arrow[r,"H"] &  \mathcal{L}^{\otimes m} \otimes \O_{X}/\I^{r}  \arrow[r] &
(\mathcal{L}^{\otimes m} \otimes \O_{X}/\I^{r} )|_{H} \arrow[r]& 0
\end{tikzcd}
\end{equation}
are exact for all $m \in \Z$ and $r \in \Z_{\geq 1}$.
Here the second map is the one induced by $H$.
\end{lemma}
\begin{proof}
Let $X = \bigcup_{i=1}^{t}\Spec A_{i}$ be an open affine cover.
Let $I_{i} = \I(\Spec A_{i})$.
Let $\p_{ij}$, $j=1,\dots, n_{i}$ be prime ideals of $A_{i}$ obtained by
applying \cref{lem:quotreg} to $I_{i} \subset A_{i}$.
Let $Z_{ij} \subset X$ be the irreducible closed subset with generic point $\p_{ij} \in \Spec A_{i}$.
Then for $H \in \dd$ such that $Z_{ij} \not\subset H$ for all $i,j$, the sequence
\begin{equation}
\begin{tikzcd}[column sep=small] 
0 \arrow[r] &  \O_{X}/\I^{r} \arrow[r,"H"] &  \mathcal{L} \otimes \O_{X}/\I^{r}  \arrow[r] &
(\mathcal{L}\otimes \O_{X}/\I^{r} )|_{H} \arrow[r]& 0
\end{tikzcd}
\end{equation}
is exact.
By tensoring $ \mathcal{L}^{\otimes m-1}$, we get the desired exactness.
\end{proof}

\begin{proof}[Proof of \cref{lem:glsecthick}]
We prove the statement by induction on $l$.
When $l=0$, 
\begin{align}
h^{0}(X, \L^{\otimes m} \otimes \O_{X}/\I^{r}) = h^{0}(X, \O_{X}/\I^{r}) 
\end{align}
and we are done.

Suppose $l \geq 1$.
Note that in this case we have $N = \dim X \geq 2$.
By \cref{lem:generalhpsecexact} and \cite[(3.4.10), Corollary 3.4.14]{joinintersec},
for a general $H_{1} \in \dd$, we have
\begin{align}
&\text{$H_{1}$ is irreducible reduced}\\
&\text{$Y \cap H_{1}$ is pure dimension $l-1$}
\end{align}
and
\begin{equation}
\begin{tikzcd}[column sep=small] 
0 \arrow[r] & \mathcal{L}^{\otimes m-1} \otimes \O_{X}/\I^{r} \arrow[r,"H_{1}"] &  \mathcal{L}^{\otimes m} \otimes \O_{X}/\I^{r}  \arrow[r] &
(\mathcal{L}^{\otimes m} \otimes \O_{X}/\I^{r} )|_{H_{1}} \arrow[r]& 0
\end{tikzcd}
\end{equation}
is exact for all $m \in \Z$ and $r \in \Z_{\geq 1}$.
Then we have
\begin{align}
h^{0}(X, \L^{ \otimes m}  \otimes \O_{X}/\I^{r}) \leq
h^{0}(X, \L^{ \otimes m-1}  \otimes \O_{X}/\I^{r}) + h^{0}(H_{1}, \L|_{H_{1}}^{ \otimes m} \otimes \O_{H_{1}} / \I_{1}^{r})
\end{align}
where $\I_{1}$ is the ideal sheaf of $Y \cap H_{1} \subset H_{1}$.
Using this repeatedly, we get
\begin{align}
h^{0}(X, \L^{ \otimes m}  \otimes \O_{X}/\I^{r}) \leq
h^{0}(X, \O_{X}/\I^{r}) + \sum_{j=1}^{m} h^{0}(H_{1}, \L|_{H_{1}}^{ \otimes j} \otimes \O_{H_{1}}/\I_{1}^{r})
\label{eq:h0Lmboundinduction}
\end{align}
for $m \geq 1$ and $r \geq 1$.
Now we apply our induction hypothesis to
$Y \cap H_{1} \subset H_{1}$, $\L|_{H_{1}}$, and the image of $\dd$ by
$H^{0}(X, \L) \longrightarrow H^{0}(H_{1}, \L|_{H_{1}})$.
Then for a general sequence $H_{2}, \dots , H_{l} \in \dd$,
we have
\begin{align}
&\text{$H_{1\dots i}=H_{1}\cap \cdots \cap H_{i}$ is irreducible and reduced of dimension $N-i$}\\
&\text{$Y \cap H_{1\dots i}$ is pure dimension $l-i$}
\end{align}
and
\begin{align}
h^{0}(H_{1}, \L|_{H_{1}}^{ \otimes j} \otimes \O_{H_{1}}/\I_{1}^{r})
\leq 
\sum_{i=0}^{l-1} \binom{j+i-1}{i} h^{0}(H_{1\cdots (i+1)}, \O_{H_{1\cdots (i+1)}}/\I_{i+1}^{r})
\end{align}
for $j \geq 1$ and $r \geq 1$,
where $\I_{i+1}$ is the ideal sheaf of $Y \cap H_{1\cdots (i+1)} \subset H_{1\cdots (i+1)}$.
Plugging into \cref{eq:h0Lmboundinduction}, we get
\begin{align}
&h^{0}(X, \L^{ \otimes m}  \otimes \O_{X}/\I^{r} ) \\
&\leq 
h^{0}(X, \O_{X}/\I^{r}) + 
\sum_{i=0}^{l-1} \sum_{j=1}^{m}\binom{j+i-1}{i} h^{0}(H_{1\cdots (i+1)}, \O_{H_{1\cdots (i+1)}}/\I_{i+1}^{r})\\
&=
h^{0}(X, \O_{X}/\I^{r}) + 
\sum_{i=0}^{l-1} \binom{m+i}{i+1} h^{0}(H_{1\cdots (i+1)}, \O_{H_{1\cdots (i+1)}}/\I_{i+1}^{r})\\
&=
h^{0}(X, \O_{X}/\I^{r}) + 
\sum_{i=1}^{l} \binom{m+i-1}{i} h^{0}(H_{1\cdots i}, \O_{H_{1\cdots i}}/\I_{i}^{r})
\end{align}
and we are done.
\end{proof}

\begin{lemma}\label{lem:relatetobup}
Let $X$ be a projective variety over a field $k$.
Let $Y \subset X$ be a closed subscheme with ideal sheaf $ \I \subset \O_{X}$.
Let $ \pi \colon \widetilde{X} \longrightarrow X$ be the blow up along $Y$, 
and set $E = \pi^{-1}(Y) = \Proj \oplus_{n\geq 0} \I^{n}/\I^{n+1}$ be the exceptional divisor.
Then there is $n_{0} \geq 0$ such that for all $n \geq n_{0}$, we have
\begin{align}
\I^{n}/\I^{n+1} \simeq (\pi|_{E})_{*}\O_{E}(n)
\end{align}
where $\pi|_{E}= E \longrightarrow Y$ is the restriction of $\pi$.
In particular, we have $h^{0}(Y, \I^{n}/\I^{n+1} ) = h^{0}(E,\O_{E}(n))$ and hence 
\begin{align}
h^{0}(X, \O_{X}/\I^{n})  \leq h^{0}(X, \O_{X}/\I^{n_{0}}) + \sum_{i=n_{0}}^{n-1} h^{0}(E, \O_{E}(i))
\end{align}
for all $n \geq n_{0}$.
\end{lemma}
\begin{proof}
There is a canonical morphism 
$\I^{n}/\I^{n+1} \longrightarrow (\pi|_{E})_{*}\O_{E}(n)$
by the definition of relative $\Proj$ and \cite[II Exercise 5.9 (a)]{HartshorneAG}.
It is isomorphism for all large $n$, say at least $n_{0}$, by \cite[II Exercise 5.9 (b)]{HartshorneAG}.
For the second statement, consider the exact sequence 
\begin{equation}
\begin{tikzcd}
0 \arrow[r] & \I^{n}/\I^{n+1} \arrow[r] & \O_{X}/\I^{n+1} \arrow[r] & \O_{X}/\I^{n} \arrow[r] & 0.
\end{tikzcd}
\end{equation}
By taking global section, we get
\begin{align}
h^{0}(X, \O_{X}/\I^{n+1}  ) - h^{0}(X, \O_{X}/\I^{n}  ) &\leq h^{0}(X, \I^{n}/\I^{n+1} )\\
& = h^{0}(Y, \I^{n}/\I^{n+1}) = h^{0}(E, \O_{E}(n))  
\end{align}
for $n \geq n_{0}$.
Taking sum of these inequalities, we get
\begin{align}
h^{0}(X, \O_{X}/\I^{n}) - h^{0}(X, \O_{X}/\I^{n_{0}}) \leq \sum_{i=n_{0}}^{n-1} h^{0}(E, \O_{E}(i)).
\end{align}
\end{proof}

\begin{definition}
Let $X$ be a projective scheme over a field $k$.
Let $Y \subset X$ be a closed subscheme with ideal sheaf $ \I \subset \O_{X}$.
Let $\I \subset \O_{X}$ be the ideal sheaf of $Y \subset X$.
We set
\begin{align}
\tau(Y, X) := \limsup_{r \to \infty} \frac{h^{0}(X, \O_{X} / \I^{r})}{r^{\dim X}}.
\end{align}
This is finite by \cref{lem:relatetobup}, and \cite[Example 1.2.33]{Lazpos1} or \cref{lem:arr} below.
\end{definition}

\begin{proposition}\label{prop:existencevansectau}
Let $k$ be an algebraically closed field of characteristic zero.
Let $X$ be a projective variety over $k$ of dimension $N \geq 0$.
Let $Y \subset X$ be a proper closed subscheme of pure dimension $l$ with ideal sheaf $ \I \subset \O_{X}$.
Let $ \mathcal{L}$ be a very ample invertible $\O_{X}$-module on $X$.
Consider a general sequence $H_{1}, \dots, H_{l} \in | \mathcal{L}|$ such that the conclusions of \cref{lem:glsecthick} hold.
Set 
\begin{align}
\tau_{i} = \tau(Y\cap  H_{1\cdots i} ,  H_{1\cdots i})
\end{align}
for $i=0,\dots, l$.
Let $\e_{0} > 0$.
Then there is $a \in \Z_{\geq 1}$ such that for all $m,r \geq a$ with
\begin{align}
\sum_{i=0}^{l} \frac{\tau_{i}}{i !} \bigg( \frac{r}{m} \bigg)^{N-i} \leq \frac{ (c_{1}( \mathcal{L})^{N}) }{N!} - \e_{0}, \label{eq:condonrm}
\end{align}
we have
\begin{align}
H^{0}(X, \mathcal{L}^{\otimes m} \otimes \I^{r}) \neq 0.
\end{align}
\end{proposition}
\begin{proof}
First of all, the statement is trivial if $Y =  \emptyset$.
Thus we may assume $Y \neq  \emptyset$.
Also, we may assume the right-hand-side of \cref{eq:condonrm} is positive.

By the exact sequence
\begin{equation}
\begin{tikzcd}
0  \arrow[r] & \L^{ \otimes m} \otimes \I^{r}  \arrow[r] & \L^{m} \arrow[r] & \L^{ \otimes m} \otimes \O_{X}/\I^{r} \arrow[r] & 0,
\end{tikzcd}
\end{equation}
we have
\begin{align}
h^{0}(X,  \L^{ \otimes m} \otimes \I^{r} ) \geq h^{0}(X, \L^{m}) - h^{0}(X, \L^{ \otimes m} \otimes \O_{X}/\I^{r}) \label{eq:h0lowerbound}
\end{align}
for all $m \in \Z$ and $r \in \Z_{\geq 0}$.
Let us fix small $\e > 0$. We will specify how small it should be later.
Since
\begin{align}
\lim_{m \to \infty} \frac{h^{0}(X, \L^{ \otimes m})}{m^{N}/N!} = \vol( \L) = (c_{1}(\L)^{N})
\end{align}
(cf. \cite[Example 11.4.7]{Lazpos2}), 
there is $m(\e) \geq 1$ such that
\begin{align}
h^{0}(X, \L^{ \otimes m}) > \bigg( \frac{(c_{1}(\L)^{N})}{N!} - \e \bigg) m^{N} \label{eq:h0Lmlowerbound}
\end{align}
for all $m \geq m(\e)$.

On the other hand, since
\begin{align}
\tau_{i} = \limsup_{r \to \infty} \frac{h^{0}(H_{1\cdots i}, \O_{H_{1\cdots i}}/\I_{i}^{r})}{r^{N-i}},
\end{align}
there is $r(\e) \geq 1$ such that
\begin{align}
h^{0}(H_{1\cdots i}, \O_{H_{1\cdots i}}/\I_{i}^{r}) \leq (\tau_{i} + \e)r^{N-i}
\end{align}
for all $r \geq r(\e)$ and $i=0,\dots, l$.
Thus by \cref{lem:glsecthick}, we have
\begin{align}
h^{0}(X, \L^{ \otimes m} \otimes \O_{X}/\I^{r})
&\leq
\sum_{i=0}^{l} \binom{m+i-1}{i} h^{0}(H_{1\cdots i}, \O_{H_{1\cdots i}}/\I_{i}^{r})\\
&\leq
\sum_{i=0}^{l} \binom{m+i-1}{i}  (\tau_{i} + \e)r^{N-i} \label{eq:h0LmOIrupperbound}
\end{align}
for all $m \geq 1$ and $r \geq r(\e)$.

By \cref{eq:h0lowerbound}, \cref{eq:h0Lmlowerbound}, and \cref{eq:h0LmOIrupperbound}, 
$H^{0}(X, \L^{ \otimes m} \otimes \I^{r}) \neq 0$ if
$m \geq m(\e)$, $r\geq r(\e)$, and
\begin{align}
\sum_{i=0}^{l} \binom{m+i-1}{i}  (\tau_{i} + \e)r^{N-i}  \leq \bigg( \frac{(c_{1}(\L)^{N})}{N!} - \e \bigg) m^{N}. \label{eq:rmineq1}
\end{align}
By enlarging $m(\e)$ if necessary, we may assume 
\begin{align}
\frac{1}{m^{i}}\binom{m+i-1}{i}  \leq \frac{1}{i!} + \e
\end{align}
for all $m \geq m(\e)$ and $i=0,\dots, l$.
Then \cref{eq:rmineq1} follows from
\begin{align}
\sum_{i=0}^{l} \bigg( \frac{1}{i!} + \e \bigg) (\tau_{i} + \e)\bigg( \frac{r}{m}\bigg)^{N-i}  \leq \frac{(c_{1}(\L)^{N})}{N!} - \e .
\end{align}

Now, note that $\tau_{l} > 0$ since $Y \cap H_{1\cdots l}$ is non-empty and $0$-dimensional.
Thus the condition \cref{eq:condonrm} implies 
\begin{align}
\frac{r}{m} \leq \bigg\{ \frac{l!}{\tau_{l}} \bigg( \frac{(c_{1}(\L)^{N})}{N!} - \e_{0}\bigg) \bigg\}^{1/(N-l)} =: M.
\end{align}
Hence if $m,r \geq 1$ satisfy condition \cref{eq:condonrm}, then
\begin{align}
\sum_{i=0}^{l} \bigg( \frac{1}{i!} + \e \bigg) (\tau_{i} + \e)\bigg( \frac{r}{m}\bigg)^{N-i} 
\leq
\frac{ (c_{1}( \mathcal{L})^{N}) }{N!} - \e_{0} + 
\e \sum_{i=0}^{l} \bigg( \frac{1}{i!} + \tau_{i} + \e\bigg)M^{N-i}.
\end{align}
Therefore, if we took $\e$ so that
\begin{align}
\frac{ (c_{1}( \mathcal{L})^{N}) }{N!} - \e_{0} + 
\e \sum_{i=0}^{l} \bigg( \frac{1}{i!} + \tau_{i} + \e\bigg)M^{N-i}
\leq
\frac{(c_{1}(\L)^{N})}{N!} - \e,
\end{align}
then $a = \max\{m(\e), r(\e)\}$ has the desired property.
\end{proof}

Next, we are going to find an upper bound of ``$\tau_{i}$" in dynamical situation, i.e.\ 
upper bound of ``$\tau(f^{-n}(Y)\cap H_{1\cdots i}, H_{1 \cdots i})$".
It will be done in \cref{prop:boundoftaudyn}. We prepare two lemmas.

\begin{lemma}\label{lem:arr}
Let $N \in \Z_{\geq 0}$.
Then there is $C > 0$ with the following property.
Let 
\begin{itemize}
\item $k$ be an infinite field;
\item $X$ be a projective scheme over $k$ of pure dimension $N$;
\item $ \mathcal{L}$ be an invertible $\O_{X}$-module on $X$;
\item $ \mathcal{L}_{1}, \mathcal{L}_{2}$ be base point free  invertible $\O_{X}$-module such that 
$ \mathcal{L} \simeq \mathcal{L}_{1} \otimes \mathcal{L}_{2}^{-1} $;
\item $ \mathcal{F}$ be a coherent $\O_{X}$-module.
\end{itemize}
We set 
\begin{align}
\d_{ \mathcal{L}_{1}, \mathcal{L}_{2}} = \max\{ (c_{1}( \mathcal{L}_{1})^{i} \cap c_{1}( \mathcal{L}_{2})^{\dim X - i} \cap [X]) \mid 0 \leq i \leq \dim X  \}.
\end{align}
Then we have
\begin{align}
h^{0}(X, \mathcal{F} \otimes \mathcal{L}^{\otimes n}) \leq C \rank(\mathcal{F})\d_{ \mathcal{L}_{1} , \mathcal{L}_{2}}  n^{\dim X} + O(n^{\dim X - 1})
\end{align}
for $n \geq 0$,
where the implicit constant depends at most on $X, \mathcal{L}, \mathcal{L}_{1}, \mathcal{L}_{2}$, and $ \mathcal{F}$.
Here 
\begin{align}
\rank( \mathcal{F}) := \max \{ \dim_{\k(\eta)} \mathcal{F} \otimes \k(\eta) \mid \text{$\eta \in X$ generic point } \}
\end{align}
where $\k(\eta)$ is the residue field at $\eta$.
\end{lemma}

\begin{remark}
The cycle $[X]$ is the fundamental cycle, i.e.
\begin{align}
[X] = \sum_{\text{$\eta \in X$ generic point}} l_{\O_{X,\eta}}(\O_{X,\eta}) [ \overline{\{\eta\}} ]
\end{align}
where $l_{\O_{X,\eta}}(\O_{X,\eta})$ is the length of the local ring $\O_{X,\eta}$.
For a cycle class $\a$ on $X$, $(\a)$ is the push-forward of $\a$ to $\Spec k$, where
the group of zero-cycle classes $A_{0}(\Spec k)$ is identified with $\Z$ so that $[\Spec k]$ corresponds to $1$.
\end{remark}

\begin{proof}
We prove by induction on $N$.
When $N = 0$, there is a surjection $\O_{X}^{ \oplus \rank \mathcal{F}} \longrightarrow \mathcal{F}$.
Also, we have $\L \simeq \O_{X}$, and hence
\begin{align}
h^{0}(X, \mathcal{F} \otimes \L^{ \otimes n}) &= h^{0}(X, \mathcal{F}) \leq \rank (\mathcal{F}) h^{0}(X, \O_{X}) \\
&= \rank (\mathcal{F}) ([X]) = \rank (\mathcal{F}) \d_{\L_{1}, \L_{2}}.
\end{align}

Suppose $N \geq 1$.
Let $U \subset X$ be an open subset on which we have a surjection
$\O_{U}^{\oplus \rank \mathcal{F}} \longrightarrow \mathcal{F}|_{U}$.
Since $\L_{1}, \L_{2}$ are base point free, there are global sections 
$s_{i} \in H^{0}(X, \L_{i}) \ (i=1,2)$ such that
they do not vanish at associated points of $X$ and $ \mathcal{F}$,
and any generic points of $X \setminus U$.
Let $A,B \subset X$ be the effective Cartier divisors defined by $s_{1},s_{2}$
respectively.
Then 
\begin{align}
&\rank \mathcal{F}|_{A}, \rank \mathcal{F}|_{B} \leq \rank \mathcal{F}
\end{align}
and
\begin{equation}
\begin{tikzcd}[row sep=small] 
0 \arrow[r] & \mathcal{F} \otimes \L_{1}^{-1} \arrow[r,"s_{1}"] & \mathcal{F} \arrow[r] & \mathcal{F}|_{A} \arrow[r] & 0\\
0 \arrow[r] & \mathcal{F} \otimes \L_{2}^{-1} \arrow[r,"s_{2}"] & \mathcal{F} \arrow[r] & \mathcal{F}|_{B} \arrow[r] & 0
\end{tikzcd}
\end{equation}
are exact.
The tensoring $\L^{ \otimes n+1}$ and $\L^{ \otimes n}$ respectively, we get the following exact sequences:
\begin{equation}
\begin{tikzcd}[row sep=small, column sep = small] 
0 \arrow[r] & \mathcal{F} \otimes  \L^{ \otimes n} \otimes \L_{2}^{-1} \arrow[r,"s_{1}"] & \mathcal{F}  \otimes \L^{ \otimes n+1} \arrow[r] & \mathcal{F}|_{A}\otimes \L|_{A}^{ \otimes n+1} \arrow[r] & 0,\\
0 \arrow[r] & \mathcal{F} \otimes  \L^{ \otimes n} \otimes \L_{2}^{-1} \arrow[r,"s_{2}"] & \mathcal{F}  \otimes \L^{ \otimes n} \arrow[r] & \mathcal{F}|_{B}\otimes \L|_{B}^{ \otimes n} \arrow[r] & 0.
\end{tikzcd}
\end{equation}
Thus we get
\begin{align}
&|h^{0}(X, \mathcal{F} \otimes \L^{ \otimes n+1}) - h^{0}(X, \mathcal{F} \otimes \L^{ \otimes n})  |\\
&\leq
h^{0}(A, \mathcal{F}|_{A}\otimes \L|_{A}^{ \otimes n+1}) + h^{0}(B,  \mathcal{F}|_{B}\otimes \L|_{B}^{ \otimes n} ).
\end{align}
Note that if $A$ or $B$ is empty, then $h^{0} = 0$ and we can simply ignore them.
Thus we may assume $A, B  \neq \emptyset$.
Since $\dim A = \dim B = \dim X - 1$, we can apply the induction hypothesis to
$A, \mathcal{F}|_{A}, \L|_{A}, \L_{1}|_{A}, \L_{2}|_{A}$ and $B, \mathcal{F}|_{B}, \L|_{B}, \L_{1}|_{B}, \L_{2}|_{B}$.
Then there is a constant $C$ depending only on $\dim X -1$ such that
\begin{align}
&|h^{0}(X, \mathcal{F} \otimes \L^{ \otimes n+1}) - h^{0}(X, \mathcal{F} \otimes \L^{ \otimes n})  |\\
&\leq
C \rank (\mathcal{F}|_{A})  \d_{\L_{1}|_{A}, \L_{2}|_{A}} (n+1)^{\dim X - 1}\\
&+C\rank (\mathcal{F}|_{B})  \d_{\L_{1}|_{B}, \L_{2}|_{B}} n^{\dim X - 1}
+O(n^{\dim X - 2})
\end{align}
where the implicit constant depends at most on $A, B, \mathcal{F}, \L, \L_{1}, \L_{2}$.
Note that
\begin{align}
&(c_{1}(\L_{1}|_{A})^{i} \cap c_{1}(\L_{2}|_{A})^{\dim X -1  - i} \cap [A]) \\
&=
(c_{1}(\L_{1})^{i} \cap c_{1}(\L_{2})^{\dim X -1  - i} \cap c_{1}(\L_{1}) \cap [X]) 
\leq \d_{\L_{1}, \L_{2}}
\end{align}
and hence $\d_{\L_{1}|_{A}, \L_{2}|_{A}} \leq \d_{L_{1}, \L_{2}}$.
Here we use pure dimensionality of $X$ to deduce $[A] = c_{1}(\L_{1}) \cap [X]$ as cycle classes on $X$
(cf. \cite[Lemma 1.7.2 or Proposition 2.6 (d)]{fulintersec}).
By the same argument, we have $\d_{\L_{1}|_{B}, \L_{2}|_{B}} \leq \d_{\L_{1}, \L_{2}}$.
Therefore we get
\begin{align}
&|h^{0}(X, \mathcal{F} \otimes \L^{ \otimes n+1}) - h^{0}(X, \mathcal{F} \otimes \L^{ \otimes n})  |\\
&\leq
2C \rank( \mathcal{F})\d_{\L_{1}, \L_{2}} n^{\dim X -1 } + O(n^{\dim X - 2})
\end{align}
where the implicit constant depends at most on $A, B, \mathcal{F}, \L, \L_{1}, \L_{2}$.
 Thus we have
 \begin{align}
& h^{0}(X, \mathcal{F} \otimes \L^{ \otimes n}) \\
& \leq
h^{0}(X, \mathcal{F})
 + \sum_{j=0}^{n-1} \big( 2C \rank( \mathcal{F})\d_{\L_{1}, \L_{2}} j^{\dim X -1 } + O(j^{\dim X - 2}) \big)\\
 &\leq
 C' \rank( \mathcal{F}) \d_{\L_{1}, \L_{2}} n^{\dim X} + O(n^{\dim X -1})
 \end{align}
where $C'$ depends at most on $\dim X$, and
the implicit constant depends at most on $A, B, \mathcal{F}, \L, \L_{1}, \L_{2}$.
By taking infimum of the implicit constants varying $A,B$, we see
that the implicit constant can be chosen depending only on $X, \mathcal{F}, \L, \L_{1}, \L_{2}$.
\end{proof}

\begin{lemma}\label{lem:hpcutregemb}
Let $k$ be an infintie field.
Let $X$ be a projective variety over $k$ of dimension $N \geq 0$.
Let $H$ be a very ample divisor on $X$ and $ \mathfrak{d} \subset |H|$ be a linear system 
defining a closed immersion.
Let $Y \subset X$ be a regular embedding of codimension $e \geq 0$.
Let $0 \leq i \leq N-e$.
Then for a general sequence $H_{1},\dots , H_{i} \in \mathfrak{d}$,
\begin{align}
Y \cap H_{1} \cap \cdots \cap H_{i} \subset  H_{1} \cap \cdots \cap H_{i} 
\end{align}
is regular embedding of codimension $e$.
\end{lemma}
\begin{proof}
When $N = 0,1$, there is nothing to prove.
Thus we may assume $N \geq 2$.
By \cite[(3.4.10), Corollary 3.4.14]{joinintersec}, $H_{1} \cap \cdots \cap H_{i}$
is a projective variety of dimension $N-i$ for a general sequence $H_{1}, \dots, H_{i} \in \dd$
if $i \leq N-1$.
Thus it is enough to prove the statement for $i = 1$ (and $1 \leq  e \leq N-1$) case.

Let $\Spec A \subset X$ be an arbitrary open subscheme
on which the ideal $I \subset A$ of $Y \cap \Spec A \subset \Spec A$
is generated by an $A$-regular sequence $a_{1},\dots, a_{e} \in I$.
Then the Koszul complex gives a resolution of $A/I$:
\begin{equation}
\begin{tikzcd}[column sep=small]
0 \arrow[r] & K_{e}(a_{1},\dots,a_{e}) \arrow[r] & \cdots \arrow[r] &  K_{0}(a_{1},\dots,a_{e}) \arrow[r] &  A/I \arrow[r] & 0 \label{eq:koszulresol}
\end{tikzcd}
\end{equation}
is exact (cf.  \cite[Theorem 16.5]{MatRing}).

In general, if we are given an exact sequence of coherent sheaves 
\begin{equation}
\begin{tikzcd}
0 \arrow[r] & \mathcal{F}' \arrow[r] & \mathcal{F} \arrow[r] & \mathcal{F}'' \arrow[r] & 0
\end{tikzcd}
\end{equation}
on an open subset $U \subset X$, then for a general $H \in \dd$, the restriction
\begin{equation}
\begin{tikzcd}
0 \arrow[r] & \mathcal{F}'|_{H \cap U} \arrow[r] & \mathcal{F}|_{H \cap U} \arrow[r] & \mathcal{F}''|_{H \cap U} \arrow[r] & 0
\end{tikzcd}
\end{equation}
is exact.
Indeed, if we pick $H \in \dd$ so that it does not contain the associated points of 
$ \mathcal{F}', \mathcal{F}, \mathcal{F}''$, then we have the commutative diagram
\begin{equation}
\begin{tikzcd}
   & 							0    	\arrow[d]		&        0 	\arrow[d]	   &                                      &\\
0 \arrow[r] & \mathcal{F}' \otimes \O_{U}(-H|_{U})   \arrow[r] \arrow[d]  & \mathcal{F}'  \arrow[r] \arrow[d]& \mathcal{F}'|_{H \cap U} \arrow[r] \arrow[d]& 0\\
0 \arrow[r] & \mathcal{F} \otimes \O_{U}(-H|_{U})   \arrow[r] \arrow[d] & \mathcal{F}  \arrow[r] \arrow[d]& \mathcal{F}|_{H \cap U} \arrow[r] \arrow[d]& 0\\
0 \arrow[r] & \mathcal{F}'' \otimes \O_{U}(-H|_{U})   \arrow[r] \arrow[d] & \mathcal{F}''  \arrow[r] \arrow[d]& \mathcal{F}''|_{H \cap U} \arrow[r] \arrow[d]& 0\\
       & 							0    			&        0 	   	   &                   0                   &
\end{tikzcd}
\end{equation} 
where all the rows and columns are exact.
Then by the snake lemma, $ \mathcal{F}'|_{H \cap U} \to \mathcal{F}|_{H \cap U}$ is injective and we are done.

Applying this to short exact sequences obtained by taking images and kernels of the morphisms in 
\cref{eq:koszulresol}, we see that for a general $H \in \dd$, the restriction of \cref{eq:koszulresol}
to $H \cap \Spec A$ is exact (and $H$ is irreducible and reduced).
If the ideal of $H \cap \Spec A$ is $\p \subset A$, this means 
\begin{align}
0 \longrightarrow & K_{e}(a_{1},\dots,a_{e}) \otimes_{A} A/\p \longrightarrow \cdots \\
&\cdots \longrightarrow   K_{0}(a_{1},\dots,a_{e}) \otimes_{A} A/\p \longrightarrow A/(I + \p) \longrightarrow 0 
\end{align}
is exact.
Thus by \cite[Theorem 16.5]{MatRing}, for any prime ideal $\q \subset A/\p$ with $IA/\p \subset \q$,
$a_{1},\dots,a_{e}$ is an $(A/\p)_{\q}$-regular sequence generating $I(A/\p)_{\q}$. 
This means $Y \cap \Spec A \cap H \subset  \Spec A \cap H$ is regular embedding of codimension $e$.
Covering $X$ by finitely many such $\Spec A$, we are done.
\end{proof}

The following is the key for the proof of our main theorem.

\begin{proposition}\label{prop:boundoftaudyn}
Let $k$ be an algebraically closed field of characteristic zero.
Let $X$ be a regular projective variety over $k$ of dimension $N \geq 0$.
Let $f \colon X \dashrightarrow X$ be a dominant rational map.
Let $Y \subset X$ be a proper closed subscheme of pure dimension $l$.
Suppose $Y \subset X$ is a regular embedding and $Y \subset X_{f}^{\rm back}$.
Let $H$ be a very ample divisor on $X$.
Then there is $C > 0$ depending only on $X, Y, H$ such that
for any $n \geq 0$, $i=0,\dots, l$, and a general sequence $H_{1}, \dots, H_{i} \in |H|$
(general subject to $n$), we have
\begin{align}
\tau(f^{-n}(Y) \cap H_{1} \cap \cdots \cap H_{i}, H_{1} \cap \cdots \cap H_{i}) \leq C ((f^{n})^{*} (H^{N-i} )\cdot H^{i}  ).
\end{align}
Here $(f^{n})^{*} (H^{N-i} )$ is the $i$-cycle class on $X$ defined as follows.
If $\pi \colon \widetilde{X} \longrightarrow X$ is a resolution of indeterminacy of $f^{n}$,
then $(f^{n})^{*} (H^{N-i} ) = \pi_{*} (f^{n} \circ \pi)^{*}(H^{N-i})$. This is independent of the choice of $\pi$.
\end{proposition}
\begin{proof}
Let $i \in \{0, \dots, l\}$ and $H_{1},\dots, H_{i} \in |H|$ be a general sequence.
Write $V = H_{1} \cap \cdots \cap H_{i}$.
By \cite[(3.4.10), Corollary 3.4.14]{joinintersec}, we may assume $V$ is a regular projective variety.
Let  $n \geq 0$ and $\I \subset \O_{V}$ be the ideal sheaf of $f^{-n}(Y) \cap V \subset V$.
(Since $Y \subset X_{f}^{\rm back}$, $f^{-n}(Y)$ is naturally defined.)
Let us consider the blow up of $V$ along $\I$:
\begin{equation}
\begin{tikzcd}
\Bl_{\I}V \arrow[d] \arrow[r,hookleftarrow] & P(C_{f^{-n}(Y)\cap V}V) \arrow[d] \arrow[r,phantom,"=:"] &[-2em]  E\\
V \arrow[r,hookleftarrow]  & f^{-n}(Y)\cap V  &
\end{tikzcd}
\end{equation}
Here $ P(C_{f^{-n}(Y)\cap V}V)$ is the projectivized normal cone to $f^{-n}(Y)\cap V$, i.e.
\begin{align}
P(C_{f^{-n}(Y)\cap V}V) = \Proj \oplus_{m \geq 0} \I^{m}/\I^{m+1}.
\end{align}
By \cref{lem:relatetobup}, there is $m_{0} \geq 1$ such that
\begin{align}
h^{0}(V, \O_{V} / \I^{m}) \leq h^{0}(V, \O_{V}/\I^{m_{0}}) + \sum_{j = m_{0}}^{m-1} h^{0}(E, \O_{E}(j)) \label{eq:h0VOVImbound}
\end{align}
for $m \geq m_{0}$.

Now we claim that there is $C > 0$ depending only on $X,Y$, and $H$ such that
\begin{align}
h^{0}(E, \O_{E}(j)) \leq C ((f^{n})^{*}(H^{N-i}) \cdot H^{i}) j^{N-i-1} + O(j^{N-i-2}) \label{eq:keyboundh0EOj}
\end{align}
for all $j \geq 1$ where the implicit constant depends at most on $X, Y, H, f, n, V$, and $E$.
Assume for the moment that \cref{eq:keyboundh0EOj} has been established.
Then plugging into \cref{eq:h0VOVImbound}, we have
\begin{align}
h^{0}(V, \O_{V}/\I^{m}) \leq h^{0}(V, \O_{V}/\I^{m_{0}}) + C ((f^{n})^{*}(H^{N-i}) \cdot H^{i}) m^{N-i} + O(m^{N-i-1})
\end{align}
for all $m \geq m_{0}$, where the implicit constant depends at most on $X, Y, H, f, n, V, E$ and independent of $m$.
Thus we get
\begin{align}
\tau(f^{-n}(Y) \cap V, V) = \limsup_{m \to \infty} \frac{h^{0}(V, \O_{V}/\I^{m})}{m^{N-i}} \leq C ((f^{n})^{*}H^{N-i} \cdot H^{i}).
\end{align}

Now we start to prove \cref{eq:keyboundh0EOj}.
Consider the following commutative diagram:
\begin{equation}
\begin{tikzcd}
E \arrow[r,phantom,"="] &[-2em] P(C_{f^{-n}(Y)\cap V} V) \arrow[d,"q"] \arrow[r] & P(C_{f^{-n}(Y)}X) \arrow[d] \arrow[r] 
& P(C_{Y}X) \arrow[d,"p"] \\
& f^{-n}(Y)\cap V \arrow[d,hook] \arrow[r,hook,"\iota"] &  f^{-n}(Y) \arrow[d,hook]  \arrow[r,"g"] & Y \arrow[d,hook] \\
& V \arrow[r,hook] & X \arrow[r,"f^{n}",dashed] & X .
\end{tikzcd}
\end{equation}
Since $\O_{P(C_{Y}X)}(1)$ is $p$-ample and $H|_{Y}$ is very ample,
we can pick and fix $k_{1},k_{2} \geq 1$ such that
\begin{align}
\O_{P(C_{Y}X)}(k_{1}) \otimes p^{*}\O_{Y}(k_{2}H|_{Y}),\ \O_{P(C_{Y}X)}(k_{1}+1) \otimes p^{*}\O_{Y}(k_{2}H|_{Y})
\end{align}
are both globally generated (cf. \cite[Part One, 1, 1.7]{Lazpos1}).
Thus 
\begin{align}
&\L :=\O_{E}(k_{1}+1) \otimes q^{*}\O_{f^{-n}(Y) \cap V}(k_{2}(g \circ \iota)^{*}H|_{Y}),\\
&\M := \O_{E}(k_{1}) \otimes q^{*}\O_{f^{-n}(Y) \cap V}(k_{2}(g \circ \iota)^{*}H|_{Y})
\end{align}
are globally generated.
Then we have $\O_{E}(1) = \L \otimes \M^{-1}$.
Note that $E$ is pure dimension $\dim V - 1 = N-i-1$ since it is an effective Cartier divisor on $\Bl_{\I}V$.
Thus by \cref{lem:arr}, there is $C > 0$ depending only on $N$ such that
\begin{align}
h^{0}(E, \O_{E}(m)) \leq C \d_{\L, \M} m^{N-i-1} + O(m^{N-i-2})
\end{align}
where the implicit constant depends at most on $E, \O_{E}(1), \L, \M$.
Here 
\begin{align}
\d_{\L,\M} = \max\{ c_{1}(\L)^{j}\cap c_{1}(\M)^{\dim E - j} \cap [E] \mid 0 \leq j \leq \dim E \}.
\end{align}
By the form of $\L, \M$, we have
\begin{align}
\d_{\L,\M} \leq C' \max\big( \{ \g_{j} \mid 0 \leq j \leq \dim E\} \cup \{0\}\big)
\end{align}
where
\begin{align}
\g_{j} = ( c_{1}(q^{*} (g \circ \iota)^{*}H|_{Y})^{j} \cap c_{1}(\O_{E}(1))^{\dim E - j} \cap [E]  ),
\end{align}
and $C' > 0$ is a constant depending only on $k_{1}, k_{2}, \dim E$.
To end the proof, it is enough to show that $\g_{j} \leq C'' ((f^{n})^{*}(H^{N-i}) \cdot H^{i})$
for some constant $C'' > 0$ depending only on $X, Y$, and $H$.

By the projection formula, we have
\begin{align}
\g_{j} & = ((g \circ \iota)^{*}H|_{Y}^{j} \cdot q_{*} ( c_{1}(\O_{E}(1))^{\dim E - j} \cap [E] ))\\
& = ((g \circ \iota)^{*}H|_{Y}^{j} \cdot  s(f^{-n}(Y)\cap V, V)_{j})
\end{align}
where $s(f^{-n}(Y)\cap V, V)_{j}$ is the $j$-dimensional part of the Segre class of closed subscheme $f^{-n}(Y) \cap V \subset V$
(cf. \cite[Corollary 4.2.2]{fulintersec}).
Now recall we assume $Y \subset X$ is a regular embedding of codimension $N - l$.
Since $f^{n}$ is flat finite over an open neighborhood of $Y$, $f^{-n}(Y) \subset X$ is also a regular embedding of codimension $N - l$.
Moreover, since $H_{1}, \dots, H_{i} \in |H|$ is a general sequence, by \cref{lem:hpcutregemb},
$f^{-n}(Y) \cap V \subset V$ is also a regular embedding of codimension $N-l$. 
Therefore, we have the following relations between normal bundles:
\begin{align}
N_{f^{-n}(Y)\cap V} V \simeq (N_{f^{-n}(Y)}X )|_{f^{-n}(Y) \cap V} \simeq (g^{*}N_{Y}X)|_{f^{-n}(Y) \cap V}.
\end{align}
Thus as cycle classes on $f^{-n}(Y)$, we have
\begin{align}
s(f^{-n}(Y) \cap V, V) &= c(N_{f^{-n}(Y)\cap V} V )^{-1} \cap [f^{-n}(Y) \cap V]\\
& = c(N_{f^{-n}(Y)}X)^{-1} \cap [f^{-n}(Y) \cap V] \\
& = c(N_{f^{-n}(Y)}X)^{-1} \cap c_{1}(H)^{i} \cap [f^{-n}(Y)]\\
& = c_{1}(H)^{i} \cap g^{*}(c(N_{Y}X)^{-1} \cap [Y]).
\end{align}
Here the first equality follows from \cite[Proposition 4.1(a)]{fulintersec}, 
where $c(-)$ stands for the total Chern class.
The second equality follows from projection formula (cf. \cite[Theorem 3.2(c)]{fulintersec}).
For the third equality, 
we use \cite[Lemma 1.7.2]{fulintersec}, 
pure dimensionality of $f^{-n}(Y)$, and that $H_{1}, \dots , H_{i}$ is a general sequence.
The last equality follows from flatness of $g$ and \cite[Theorem 3.2 (d)]{fulintersec}.

Choose a $(j+i)$-cycle $\a$ on $Y$ representing $(j+i)$-dimensional part of $c(N_{Y}X)^{-1} \cap [Y]$.
Note that this can be chosen depending only on $X, Y$.
Then we have
\begin{align}
\g_{j} = (g^{*}H|_{Y}^{j} \cdot H^{i} \cdot g^{*} \a) = ( g^{*}(c_{1}(H|_{Y})^{j} \cap \a) \cdot H^{i}).
\end{align}
Let us fix an $i$-cycle $\b$ on $Y$ representing $c_{1}(H|_{Y})^{j} \cap \a$.
Consider the graph of $f^{n}$:
\begin{equation}
\begin{tikzcd}
\G_{f^{n}} \arrow[d,"\pi",swap] \arrow[rd,"h"]& \\
X \arrow[r,dashed,"f^{n}",swap]& X
\end{tikzcd}
\end{equation}
where $h:=f^{n} \circ \pi$.

\begin{claim}\label{claim:boundingbyci}
Let $W \subset Y$ be a closed subvariety of dimension $i$.
There is $a \in \Z_{\geq 1}$ depending only on $X, H$, and $W$ such that
there are $D_{1},\dots, D_{N-i} \in |aH|$ with the following properties:
local equations of $h^{*}D_{1},\dots, h^{*}D_{N-i}$ form a regular sequence 
at every point of $h^{*}D_{1}\cap \cdots \cap h^{*}D_{N-i}$, and 
\begin{align}
[h^{*}D_{1}\cap \cdots \cap h^{*}D_{N-i} ] - g^{*}[W]
\end{align}
is an effective cycle on $\G_{f^{n}}$.
Here $g^{*}[W]$ is considered as a cycle on $\G_{f^{n}}$ via the isomorphism 
$\pi^{-1}(f^{-n}(Y)) \simeq f^{-n}(Y)$.
\end{claim}

Suppose for the moment we have proven the claim.
Let us write $\b = \sum_{t=1}^{r}m_{t}[W_{t}]$, where $m_{t} \in \Z$ and $W_{t} \subset Y$
are closed subvariety of dimension $i$.
Applying \cref{claim:boundingbyci} for each $W_{t}$,
we get corresponding $a_{t} \in \Z_{\geq 1}$ and $D_{t,1},\dots ,D_{t,N-i} \in |a_{t}H|$. 
Since the push forward by $\pi$ maps effective cycles to effective cycles, we have
\begin{align}
\g_{j} &= (g^{*}\b \cdot H^{i}) \leq \sum_{t=1}^{r}|m_{t}| (\pi_{*} [h^{*}D_{t,1}\cap \cdots \cap h^{*}D_{t,N-i} ] \cdot H^{i})\\
&=\sum_{t=1}^{r}|m_{t}| a_{t}^{N-i} (\pi_{*}(h^{*}H^{N-i} )\cdot H^{i} ) 
= \big(\sum_{t=1}^{r}|m_{t}| a_{t}^{N-i} \big) ((f^{n})^{*}(H^{N-i}) \cdot H^{i}).
\end{align}
Since $\b$ is chosen depending only on $X, Y, H$, 
$\sum_{t=1}^{r}|m_{t}| a_{t}^{N-i} $ depends only on $X, Y, H$ as well.
Thus we are done.

\begin{proof}[Proof of \cref{claim:boundingbyci}]
Let $\J \subset \O_{X}$ be the ideal sheaf of $W$ considered as a closed subscheme of $X$.
Take $a \geq 1$ so that $\J \otimes \O_{X}(aH)$ is globally generated.
Let $\dd \subset |rH|$ be the linear system defined by $H^{0}(X, \J \otimes \O_{X}(aH))$.
Let $U \subset X$ be an open subset such that
$W \subset U$, $h^{-1}(U) \longrightarrow U$ is flat finite, and $h^{-1}(U)$ is regular.
(Such $U$ exists because of the assumption $W \subset Y \subset X_{f}^{\rm back}$.)

We construct $D_{t}$'s inductively.
First take any $D_{1} \in \dd$.
Suppose we have constructed $D_{1},\dots , D_{t} \in \dd$ for some $1 \leq t \leq N-i-1$
so that local equations of $h^{*}D_{1}, \dots , h^{*}D_{t}$ form regular sequence 
at every point of $h^{*}D_{1}\cap \cdots \cap h^{*}D_{t}$.
Now we take $D_{t+1} \in \dd$ so that it does not contain any
associated points of $D_{1} \cap \cdots \cap D_{t}$ outside $W$
and images of associated points of $h^{*}D_{1}\cap \cdots \cap h^{*}D_{t}$
outside $h^{-1}(W)$.

Then the resulting sequence $D_{1}, \dots, D_{N-i}$
has the following properties:
the local equations of $h^{*}D_{1},\dots, h^{*}D_{N-i}$ form a regular sequence at every point
of $h^{*}D_{1}\cap\cdots \cap h^{*}D_{N-i} \setminus h^{-1}(W)$;
$D_{1}\cap\cdots \cap D_{N-i}$ is pure dimension $i$.
Since $X$ is regular, hence Cohen-Macaulay, the second property implies the local equations of $D_{1},\dots, D_{N-i}$
form a regular sequence at every point of $D_{1} \cap \cdots \cap D_{N-i}$.
Since $h$ is flat over $U$, 
the local equations of $h^{*}D_{1},\dots, h^{*}D_{N-i}$ form a regular sequence at every point
of $h^{*}D_{1}\cap\cdots \cap h^{*}D_{N-i} \cap h^{-1}(U)$.

Finally, since $D_{t}$ are chosen from $\dd$, we have
$h^{-1}(W) \subset h^{*}D_{1}\cap\cdots \cap h^{*}D_{N-i}$ as schemes.
Since both sides of this inclusion have pure dimension $i$,
\begin{align}
[h^{*}D_{1}\cap\cdots \cap h^{*}D_{N-i}] - [h^{-1}(W)]
\end{align}
is an effective cycle.
Since $h$ is flat over $U$ (and $W \subset X_{f}^{\rm back}$),
we have $[h^{-1}(W)] = g^{*}[W]$ and we are done.
\end{proof}
This completes the proof of \cref{prop:boundoftaudyn}.
\end{proof}


\begin{proposition}\label{prop:ggcdpullbackbound}
Let $X$ be a smooth projective variety over $\QQ$ of dimension $N$.
Let $f \colon X \dashrightarrow X$ be a dominant rational map.
Let $Y \subset X$ be a proper closed subscheme of pure dimension $l$.
Suppose $Y \subset X$ is a regular embedding and $Y \subset X_{f}^{\rm back}$.
Let $H$ be a very ample divisor on $X$.
Then for any $\e >0$, there is $n_{0} \in \Z_{\geq 0}$
such that for all $n \geq n_{0}$, there is a proper closed subset $Z \subset X$
and $\g \in \R$ such that
\begin{align}
h_{f^{-n}(Y)} \leq (d_{N-l}(f)^{1/(N-l)} + \e)^{n} h_{H} + \g
\end{align}
on $(X\setminus Z)(\QQ) $.
\end{proposition}
\begin{proof}
Let us fix arbitrary $\e>0$.
By log concavity of dynamical degrees, we have 
$d_{N-i}(f)^{1/(N-i)} \leq d_{N-l}^{1/(N-l)}$ for $i = 0,\dots, l$.
Thus there is $\e'>0$ such that
\begin{align}
( d_{N-i}(f)+ \e' )^{1/(N-i)} <  d_{N-l}(f)^{1/(N-l)}+ \frac{\e}{2} \label{eq:dyndegineq}
\end{align}
for $i=0,\dots,l$.

Let $C > 0$ be as in \cref{prop:boundoftaudyn}.
Then there is $n_{0} \in \Z_{\geq 0}$ such that for all $n \geq n_{0}$ and $i=0,\dots,l$, we have
\begin{align}
C ((f^{n})^{*}(H^{N-i}) \cdot H^{i} )  \leq ( d_{N-i}(f)+ \e')^{n}.
\end{align}

Let $\e_{0} > 0$ be such that $(H^{N})/N! - \e_{0} > 0$. 
Recalling  \cref{eq:condonrm}, we consider the following condition
\begin{align}
\sum_{i=0}^{l} \frac{( d_{N-i}(f)+ \e')^{n}}{i!}  \frac{1}{( d_{N-l}(f)^{1/(N-l)}+ \e/2 )^{n(N-i)} } \leq \frac{(H^{N})}{N!} - \e_{0}.
\end{align}
By \cref{eq:dyndegineq},  enlarging $n_{0}$ if necessary, this inequality holds for all $n \geq n_{0}$.

Now pick arbitrary $n \geq n_{0}$.
We apply \cref{prop:existencevansectau} to $X$, $f^{-n}(Y)$, $\O_{X}(H)$, and $\e_{0}$.
Let $H_{1}, \dots, H_{l} \in |H|$ be a general sequence such that the conclusions of
\cref{prop:existencevansectau}  and \cref{prop:boundoftaudyn} hold.
Let $a$ be the one obtained by \cref{prop:existencevansectau}.
Then for any integers $r,m \geq a$ such that
\begin{align}
\frac{r}{m} \leq \frac{1}{( d_{N-l}(f)^{1/(N-l)}+ \e/2 )^{n} }, \label{eq:rmadmrange}
\end{align}
we have
\begin{align}
&\sum_{i=0}^{l} \frac{\tau(f^{-n}(Y)\cap H_{1\cdots i}, H_{1\cdots i}  )}{i!} \bigg( \frac{r}{m} \bigg)^{N-i}\\
&\leq 
\sum_{i=0}^{l} \frac{( d_{N-i}(f)+ \e')^{n}}{i!}  \frac{1}{( d_{N-l}(f)^{1/(N-l)}+ \e/2 )^{n(N-i)} } \leq \frac{(H^{N})}{N!} - \e_{0},
\end{align}
where $H_{1\cdots i} = H_{1} \cap \cdots \cap H_{i}$.
Thus, if $\I \subset \O_{X}$ is the ideal sheaf of $f^{-n}(Y) \subset X$, we have
\begin{align}
H^{0}(X, \O_{X}(mH) \otimes \I^{r}) \neq 0
\end{align}
whenever $r, m \geq a$ and they satisfy \cref{eq:rmadmrange}.
Therefore there is an effective divisor $D$ linearly equivalent to $mH$ containing 
the closed subscheme defined by $\I^{r}$.
This implies 
\begin{align}
rh_{f^{-n}(Y)} \leq mh_{H} + O(1)
\end{align}
on $(X \setminus D)(\QQ)$.
Taking $r,m \geq a$ so that 
\begin{align}
\frac{m}{r} \leq (d_{N-l}(f)^{1/(N-l)} + \e)^{n},
\end{align} 
we are done.
\end{proof}

\begin{proof}[Proof of \cref{thm:gcdgrowth}]
First we remark that in the case $f$ is a morphism,
we may also assume that $Y$ is pure dimensional and the inclusion $Y \subset X$
is regular embedding. (The condition $Y \subset X_{f}^{\rm back}$ is vacuous since $X_{f}^{\rm back} = X$.)
Indeed, let $\I \subset \O_{X}$ be the ideal sheaf of $Y \subset X$.
Let $\L$ be an ample invertible $\O_{X}$-module such that
$\I \otimes \L$ is globally generated.
Let $\dd \subset |\L|$ be the linear system defined by $H^{0}(X, \I \otimes \L)$.
Then for a general sequence $D_{1},\dots , D_{N-l} \in \dd$, 
we have $Y \subset D_{1} \cap \cdots \cap D_{N-l}$ as subschemes
and $ D_{1} \cap \cdots \cap D_{N-l}$ is pure dimension $l$.
Since $X$ is smooth and hence Cohen-Macaulay, the inclusion $ D_{1} \cap \cdots \cap D_{N-l} \subset X$
is a regular embedding.
As $h_{Y} \leq h_{ D_{1} \cap \cdots \cap D_{N-l}} + O(1)$ on $(X \setminus ( D_{1} \cap \cdots \cap D_{N-l}))(\QQ)$,
replacing $Y$ with $ D_{1} \cap \cdots \cap D_{N-l}$, we may assume $Y$ is pure dimensional and
the inclusion is a regular embedding.

Let $h_{H}$ be any ample height function on $X$.
We may take it so that $h_{H} \geq 1$.
Let us take arbitrary $\e > 0$.
Since $h_{Y}$ is bounded below outside $Y$ and $O_{f}(x)$ is generic,
it is enough to show 
\begin{align}
\limsup_{n \to \infty} \frac{h_{Y}(f^{n}(x))}{h_{H}(f^{n}(x))}  \leq \e.
\end{align}

Pick and fix small $\d > 0$ and $\eta \in (0,1)$ close to $1$ such that
\begin{align}
d_{N-l}(f)^{1/(N-l)} + \d < \eta \alpha_{f}(x).
\end{align}

We apply \cite[Lemma 2.4]{matsuzawa-note-ad-dls}.
Note that by \cite[Theorem 2.2]{matsuzawa-note-ad-dls},
$ \alpha_{f}(x)$ is equal to one of the cohomological Lyapunov exponents of $f$.
Thus in the statement of \cite[Lemma 2.4]{matsuzawa-note-ad-dls}, we can take 
$\mu_{l}(f) = \alpha_{f}(x)$.
Thus there are $C > 0$, $m \in \Z_{\geq 1}$, $i_{0}, s_{0} \in \Z_{\geq 0}$ such that 
\begin{align}
h_{H}((f^{m})^{i+k}(f^{s}(x))) \geq C(\eta \alpha_{f}(x))^{mk} h_{H}((f^{m})^{i}(f^{s}(x))) \label{eq:htgrowthlowerbound}
\end{align}
for all $i \geq i_{0}, s \geq s_{0}$, and $k \geq 0$.

Take and fix $k \geq 0$ such that 
\begin{align}
\frac{ (d_{N-l}(f)^{1/(N-l)} + \d )^{mk} }{C (\eta \alpha_{f}(x))^{mk} } < \e.
\end{align}

By \cref{prop:ggcdpullbackbound}, replacing $k$ larger if necessary,
there are proper closed subset $Z \subset X$ and $\g \in \R_{\geq 0}$ such that
\begin{align}
h_{f^{-mk}(Y)} \leq (d_{N-l}(f)^{1/(N-l)} + \d )^{mk}h_{H} + \g
\end{align}
on $(X \setminus Z)(\QQ)$.

Now for large $n \geq 0$, write $n = mq + s$ with $q \in \Z$ and $s_{0} \leq s \leq s_{0} + m -1$.
Then by the functoriality of $h_{Y}$, we have
\begin{align}
h_{Y}(f^{n}(x)) &= h_{Y}(f^{mq}(f^{s}(x))) \leq h_{f^{-mk}(Y)} ((f^{m})^{q-k}(f^{s}(x))) + \g'\\
& \leq 
(d_{N-l}(f)^{1/(N-l)} + \d )^{mk}h_{H}((f^{m})^{q-k}(f^{s}(x))) + \g + \g'
\end{align}
where $\g' \in \R$ is a constant depends on $f^{mk}$ and $Y$,
provided $n$ is large enough so that
$q \geq k$ and $(f^{m})^{q-k}(f^{s}(x)) \not\in f^{-mk}(Y) \cup Z$.
Here, to get the second condition, we use the genericness of $O_{f}(x)$.

Now by \cref{eq:htgrowthlowerbound}, if $n$ is further large enough so that
$q-k \geq i_{0}$, we get
\begin{align}
\frac{h_{Y}(f^{n}(x))}{h_{H}(f^{n}(x))} 
&\leq 
\frac{ (d_{N-l}(f)^{1/(N-l)} + \d )^{mk} }{C (\eta \alpha_{f}(x))^{mk} }   +  \frac{\g + \g'}{h_{H}(f^{n}(x))}\\
&\leq 
\e +  \frac{\g + \g'}{h_{H}(f^{n}(x))}.
\end{align}
Since $h_{H}(f^{n}(x)) \to \infty$, we get
\begin{align}
\limsup_{n \to \infty} \frac{h_{Y}(f^{n}(x))}{h_{H}(f^{n}(x))}  \leq \e
\end{align}
and we are done.
\end{proof}

\section{Examples}

\begin{example}\label{ex:backnonfinCE}
Let
\begin{align}
f \colon \P^{2}_{\QQ} \dashrightarrow \P^{2}_{\QQ},\ (x:y:z) \mapsto (x^{2}y : y^{3} : z^{3}).
\end{align}
Then $f$ is a dominant rational map with $I_{f}=\{(1:0:0)\}$.
We can see $d_{1}(f) = 3$ and $d_{2}(d) = 6$.
Indeed, restricted on $ {\mathbb{G}}_{m,\QQ}^{2}$, we have the following conjugation:
\begin{equation}
\begin{tikzcd}
 {\mathbb{G}}_{m,\QQ}^{2} \arrow[r, "f"] &  {\mathbb{G}}_{m,\QQ}^{2} \\
  {\mathbb{G}}_{m,\QQ}^{2} \arrow[u] \arrow[r, "g",swap]&  {\mathbb{G}}_{m,\QQ}^{2} \arrow[u]
\end{tikzcd}
\end{equation}
where the vertical allow is $(x,y) \mapsto (xy,y)$ and $g(x,y) = (x^{2},y^{3})$.
Thus $d_{1}(f) = d_{1}(g)=3 $ and $d_{2}(f) = d_{2}(g) = 6$.
Moreover, by this conjugation, it is easy to see that for any $p,q \in \Q$ with $p,q > 1$,
the $f$-orbit of $(pq,q)$ is Zariski dense.
As the dynamical Mordell-Lang conjecture holds for endomorphisms on $\A^{2}_{\QQ}$ \cite{Xi17},
such orbits are automatically generic.

Now let us consider $Y = \{(0:0:1)\}$.
Note that $Y \not\subset (\P^{2}_{\QQ})_{f}^{\rm back}$ since $f|_{\A^{2}}^{-1}(Y) = (y=0)$ as sets.
Let $h$ be the naive height on $\P^{2}_{\QQ}$.
Let us consider a point $P = (3:2:1) \in \P^{2}(\QQ)$.
Then its $f$-orbit is generic.
Explicitly, we have $f^{n}(P) = (3^{2^{n}}2^{3^{n}-2^{n}} : 2^{3^{n}} : 1)$.
Thus
$h(f^{n}(P)) = \log 3^{2^{n}}2^{3^{n}-2^{n}} $ and $ \alpha_{f}(P) = \lim_{n \to \infty} h(f^{n}(P))^{1/n} = 3 > \sqrt{6} = d_{2}(f)^{1/2}$.
But we have
\begin{align}
\frac{h_{Y}(f^{n}(P))}{h(f^{n}(P))} = \frac{\log 2^{3^{n}-2^{n}}}{\log 3^{2^{n}}2^{3^{n}-2^{n}} } \longrightarrow 1
\end{align}
as $n \to \infty$.
This example shows 
the condition $Y \subset X_{f}^{\rm back}$
is necessary in \cref{thm:gcdgrowth}.
\end{example}

\begin{example}\label{ex:A2}
Let
\begin{align}
f \colon \P^{2}_{\QQ} \dashrightarrow \P^{2}_{\QQ},\ (x:y:z) \mapsto (x^{2}y : y^{3} + x^{2}y+xz^{2} : z^{3}).
\end{align}
Then $f$ is a dominant rational map with $I_{f}=\{(1:0:0)\}$.
We can see $d_{1}(f) = 3$ and $d_{2}(f) = 7$.
Moreover if we identify $\A^{2}_{\QQ} = (z \neq 0) \subset \P^{2}_{\QQ}$,
$f|_{\A^{2}_{\QQ}} \colon \A^{2}_{\QQ} \longrightarrow \A^{2}_{\QQ}$ is a finite morphism.
Therefore, we get $(\P^{2}_{\QQ})_{f}^{\rm back} = \A^{2}_{\QQ}$.
Thus for any $0$-dimensional $Y \subset \A^{2}_{\QQ}$, and $x \in \A^{2}(\QQ)$ with
$O_{f}(x)$ being generic and $ \alpha_{f}(x) > \sqrt{7}$, we have
\begin{align}
\lim_{n \to \infty} \frac{h_{Y}(f^{n}(x))}{h_{H}(f^{n}(x))} = 0,
\end{align}
where $h_{H}$ is any ample height on $\P^{2}_{\QQ}$.
Note that genericness follows from Zariski density since the dynamical Mordell-Lang conjecture 
is known for $\A^{2}_{\QQ}$ \cite{Xi17}.
Note also that if $Y$ is not reduced, we may replace it with its reduced scheme to prove the limit is zero.
If $Y$ is reduced, the inclusion into $\P^{2}_{\QQ}$ is automatically a regular embedding.
\end{example}

\begin{example}
Let $N \geq 2$.
Let $A = (a_{ij}) \in M_{N}(\Z)$ be an $N \times N$ integer matrix with $\det A \neq 0$.
Then $A$ induces a morphism 
\begin{align}
{\mathbb{G}}_{m,\QQ}^{N} \longrightarrow {\mathbb{G}}_{m,\QQ}^{N},
(x_{1} ,\dots , x_{N}) \mapsto (x_{1}^{a_{11}} \cdots x_{N}^{a_{1N}}, \dots x_{1}^{a_{N1}} \cdots x_{N}^{a_{NN}}).
\end{align}
Let $f \colon \P^{N}_{\QQ} \dashrightarrow \P^{N}_{\QQ}$
be the induced rational map, where we identify 
${\mathbb{G}}_{m,\QQ}^{N} \hookrightarrow \P^{N}_{\QQ}, (x_{1},\dots, x_{N}) \mapsto (x_{1}:\cdots : x_{N} : 1)$.
Let us check if the assumptions of \cref{thm:gcdgrowth} hold for this map.

The dynamical degrees of $f$ is calculated as follows.
Let $\l_{1}, \dots , \l_{N}$ be the eigenvalues of $A$ counted with multiplicity and ordered so that
$|\l_{1}| \geq \cdots \geq |\l_{N}|$.
Then we have $d_{i}(f) = |\l_{1} \cdots \l_{i}|$ \cite[Theorem 1]{Lin12dyndeg}.

By \cite{Sil12}, Kawaguchi-Silverman conjecture is known for $f|_{ {\mathbb{G}}_{m,\QQ}^{N}}$.
That is, for any point $x \in {\mathbb{G}}_{m, \QQ}^{N}$ with Zariski dense orbit, 
we have $ \alpha_{f}(x) = d_{1}(f)$.
Moreover, since $f$ is \'etale on $ {\mathbb{G}}_{m,\QQ}^{N}$, the dynamical Mordell-Lang conjecture 
holds for it and thus Zariski dense orbit is automatically generic.

Finally, as $f|_{ {\mathbb{G}}_{m,\QQ}^{N}} \colon  {\mathbb{G}}_{m,\QQ}^{N} \longrightarrow  {\mathbb{G}}_{m,\QQ}^{N}$
is finite, ${\mathbb{G}}_{m,\QQ}^{N}  \subset (\P^{N}_{\QQ})_{f}^{\rm back}$.

Therefore if $|\l_{1}| > |\l_{N}|$, for any point $x \in {\mathbb{G}}_{m,\QQ}^{N} (\QQ)$ with Zariski dense orbit,
we have
\begin{align}
d_{N}(f)^{1/N} = |\l_{1} \cdots \l_{N}|^{1/N} < |\l_{1}| = d_{1}(f)  = \alpha_{f}(x).
\end{align}
Thus for any $0$-dimensional $Y \subset {\mathbb{G}}_{m,\QQ}^{N} $, we have
\begin{align}
\lim_{n \to \infty} \frac{h_{Y}(f^{n}(x))}{h_{H}(f^{n}(x))} = 0,
\end{align}
where $h_{H}$ is any ample height on $\P^{N}_{\QQ}$.
Note we can use the same argument as in \cref{ex:A2} to treat non-reduced $Y$.
\end{example}

When constructing examples, finding explicit generic or Zariski dense orbit is 
often the most difficult part.
Here is one fact that might be helpful.

\begin{proposition}
Let $X$ be a smooth quasi-projective variety over $\QQ$ of dimension $N \geq 2$
and $f \colon X \longrightarrow X$ a quasi-finite morphism.
Suppose $f$ is $1$-cohomologically hyperbolic, i.e.\ $d_{1}(f) > d_{2}(f)$.
Then for any point $x \in X(\QQ)$, if $ \overline{\alpha}_{f}(x) = d_{1}(f)$, then the orbit $O_{f}(x)$
is Zariski dense. 
In particular,  if the dynamical Mordell-Lang conjecture holds for $f$ (e.g.\ if $f$ is \'etale or $X = \A^{2}_{\QQ}$), 
then $O_{f}(x)$ is generic.
\end{proposition}
Here $ \overline{\alpha}_{f}(x) := \limsup_{n \to \infty} \max\{1 , h_{H}(f^{n}(x))\}^{1/n}$
for any ample height function $h_{H}$ on a projectivization of $X$.
\begin{proof}
The proof is a minor modification of the proof of \cite[Proposition 4.6]{MW22}.
Suppose $O_{f}(x)$ is not Zariski dense.
Then there is an irreducible component $Z$ of $ \overline{O_{f}(x)}$ such that
$f^{k}(Z) \subset Z$ and $f^{k} \colon Z \longrightarrow Z$ is dominant for some $k \geq 1$.
Let $s \geq 0$ be such that $f^{s}(x) \in Z$.
Then
\begin{align}
\overline{\alpha}_{f}(x) = \overline{\alpha}_{f^{k}}(f^{s}(x))^{1/k}  \leq d_{1}(f^{k}|_{Z})^{1/k} \leq d_{1 + \codim Z}(f^{k})^{1/k} < d_{1}(f),
\end{align}
which is a contradiction.
Here for the first equality, we use \cite[Lemma 2.7]{MW22}.
For the first inequality, see \cite[Proposition 2.6]{MW22}.
The second inequality follows from \cite[Lemma 2.3]{MW22}.
The last inequality follows from the fact $d_{1 + \codim Z}(f^{k})^{1/k} = d_{1 + \codim Z}(f)$, the assumption $d_{2}(f) < d_{1}(f)$,
and the log concavity of dynamical degrees.
\end{proof}

\begin{example}
Let $f \colon X \dashrightarrow X$ be a birational map on a smooth projective variety $X$ over $\QQ$
of dimension $N \geq 3$.
Let $Y \subset X$ be a smooth subvariety of dimension $l$.
Let $U, V \subset X$ be non-empty open subset such that $f$ induces an isomorphism $U \longrightarrow V$.
If $f^{-n}(Y) \subset V$ for all $n \geq 0$, then $Y \subset X_{f}^{\rm back}$.
For example, if $f$ is a pseudo-automorphism (i.e.\ isomorphic in codimension one), then
this could hold for reasonably many $Y$ even if $\dim Y > 0$.
For examples of such maps, see \cite{Bl13,PZ14,BK14,OT15} for instance.

In this case, for any $x \in X_{f}(\QQ)$ such that $O_{f}(x)$ is generic and
$d_{N-l}(f)^{1/(N-l)} < \alpha_{f}(x) $, we have
\begin{align}
\lim_{n \to \infty} \frac{h_{Y}(f^{n}(x))}{h_{H}(f^{n}(x))} = 0,
\end{align}
where $h_{H}$ is any ample height on $X$.
\end{example}

\begin{question}
Let $N \geq 3$ and $1 \leq  l \leq N-2 $.
Is there a dominant rational map $f \colon \P^{N}_{\QQ} \dashrightarrow \P^{N}_{\QQ}$
with the following properties?
\begin{itemize}
\item
There are open subsets $U, V \subset \P^{N}_{\QQ}$ such that
$\codim (\P^{N}_{\QQ} \setminus V) \geq l+1$, 
$U \cap I_{f} =  \emptyset$, 
and $f|_{U} \colon U \longrightarrow V$ is finite.

\item
$d_{N-l}(f)^{1/(N-l)} < d_{1}(f)$.
\end{itemize}
For such a map, it is reasonable to expect that we can apply \cref{thm:gcdgrowth}
in a meaningful way.
\end{question}

\bibliographystyle{acm}
\bibliography{growth_ggcd}

\end{document}